\documentclass[aos,preprint]{imsart}

\RequirePackage[OT1]{fontenc}
\RequirePackage{amsthm,amsmath,mathrsfs}
\RequirePackage{amssymb}
\RequirePackage{graphicx}
\RequirePackage{bbm}
\RequirePackage[numbers]{natbib}
\RequirePackage{multirow}
\RequirePackage[colorlinks,citecolor=blue,urlcolor=blue]{hyperref}

\startlocaldefs
\numberwithin{equation}{section}
\theoremstyle{plain}

\newtheorem{theorem}{Theorem}[section]
\newtheorem{lemma}{Lemma}[section]
\newtheorem{corollary}{Corollary}[section]

\newtheorem{definition}{Definition}[section]

\theoremstyle{definition}
\newtheorem{remark}{Remark}[section]
\newtheorem{example}{Example}[subsection]

\endlocaldefs

\begin{document}

\begin{frontmatter}
\title{Multivariate Density Estimation via Adaptive
  Partitioning (I): Sieve MLE}
\runtitle{Density Estimation Via Adaptive Partitioning}

\begin{aug}
\author{\fnms{Linxi} \snm{Liu}\thanksref{t2,m1}\ead[label=e1]{linxiliu@stanford.edu}}
\and
\author{\fnms{Wing Hung} \snm{Wong}\thanksref{t2,m1,m2}\ead[label=e2]{whwong@stanford.edu}}

\thankstext{t2}{Supported by NIH grant R01GM109836, and NSF grants DMS1330132 and DMS1407557.}
\runauthor{L. Liu and W. H. Wong}

\affiliation{Department of Statistics, Stanford University\thanksmark{m1}\\Department of Health Research and Policy, Stanford University\thanksmark{m2}}

\address{Department of Statistics\\
Stanford University\\
390 Serra Mall, Sequoia Hall\\
Stanford, California 94305\\
USA\\
\printead{e1} \\
\phantom{E-mail:\ }\printead*{e2}}

\end{aug}

\begin{abstract}
We study a non-parametric approach to
multivariate density estimation. The estimators are piecewise constant density functions supported by binary partitions. The partition of the sample space is
learned by maximizing the likelihood of the corresponding histogram
on that partition. We
analyze the convergence rate of the sieve maximum likelihood estimator, and reach a conclusion
that for a relatively rich class of density functions the rate does not directly depend on the dimension. This suggests
that, under certain conditions, this method is immune to the curse of
dimensionality, in the sense that it is possible to get close to the
parametric rate even in high dimensions. We also apply
this method to several special cases, and
calculate the explicit convergence rates respectively.
\end{abstract}



\end{frontmatter}

\section{Introduction}
Density estimation is
a fundamental problem in statistics. Once an explicit estimate
of the density function is obtained, various kinds of
statistical inference can follow, including non-parametric testing, clustering,
and data compression. Previous research focused on both parametric and nonparametric density estimation methods. However, currently, increasing
dimension and data size impose great difficulty on these
traditional methods. For instance, a fixed parametric family, such as multivariate Gaussian, may fail to
capture the spatial features of the true density function under high
dimensions. On the other hand, a traditional nonparametric method, like the kernel density
estimator, may suffer from the difficulty of choosing appropriate
bandwidths \cite{jones1996}. In this paper, we
study a nonparametric method for multivariate
density estimation. This is a sieve maximum likelihood
method which employs simple, but still flexible, binary
partitions to adapt to the data distribution. In the paper, we will
carry out thorough analyses of the convergence rate to quantify the performance of this method as the
dimension increases and the regularity of the true density function varies. These
analyses demonstrate the major advantages of this method, especially
when the dimension is moderately large (e.g. 5 to 50).

\subsection{Challenges in multivariate density estimation}
Most of the established methods for density estimation were initially
designed for the estimation of univariate or low-dimensional density
functions. For example, the popular kernel method (\cite{rosenblatt1956}
and \cite{parzen1962}), which
approximates the density by the superposition of windowed kernel
functions centering on the observed data points, works well for
estimating smooth low-dimensional densities. As the dimension
increases, the accuracy of the kernel estimates becomes very sensitive
to the choice of the window size and the shape of the kernel. To obtain
good performance, both of these choices need to depend on the
data. However, the question of how to adapt these parameters to the
data has not been adequately addressed. This is especially the case for the
kernel which  itself is a multidimensional function. As a result, the
performance of current kernel estimators deteriorates rapidly as
the dimension increases.

The difficulty caused by high dimensionality is also revealed in a classic
result by Charles Stone \cite{stone1980}. In this paper, it was shown
that the optimal rate of convergence for density in $d$-dimensional
space, when the density is assumed to have $p$ bounded derivatives, is
of the order $n^{-\alpha}$, where $\alpha = p/(2p+d)$. When $d$ is small
and the density is smooth (i.e. $p$ is large), then methods such as
kernel density estimation can achieve a convergence rate almost as
good as the parametric rate of $n^{-1/2}$. However, when $d$ is large,
then even if the density has many bounded derivatives, the best
possible rate will still be unacceptably slow. Thus standard
smoothness assumptions on the density will not protect us from the
``curse of dimensionality''. Instead, we must seek alternative
conditions on the underlying class of densities that are general enough
to cover some useful applications under high dimensions, and yet strong
enough to enable the construction of density estimators with fast
convergence. More specifically, suppose $r$ is a parameter that controls
the complexity (in a sense to be made precise) of the density class,
with large value of $r$ indicating low complexity. We would like to
construct density estimators with a convergence rate of the order
$n^{-\gamma(r)}$, where $\gamma(\cdot)$ is an increasing function not
as sensitive to $d$ as that of the traditional methods, and satisfying the property that $\gamma(r) \uparrow
\frac{1}{2}$ as $r \uparrow \infty$. Since this rate is not sensitive
to $d$, it is possible to obtain fast convergence even in high
dimensional cases. For density estimators based on adaptive
partitioning, such a result is established in Theorem
\ref{th:convergencerate} below.  

\subsection{Adaptive partitioning}
\label{subsec:reviewofmethod}
The most basic method for density estimation is the
histogram. With
appropriately chosen bin width, the histogram density value
within each bin is proportional to the relative frequency of the data points in that
bin. Further developments of the method allow the bins to
depend on data, and substantial improvement can be obtained by such ``data-adaptive'' histograms (\cite{scott1979}). 

This idea has been naturally extended to multivariate
cases. Multivariate histograms with data-adaptive partitions have been
studied in \cite{shang1994} and \cite{Ooi2002}. The breakthrough work
of Lugosi and Nobel \cite{lugosi1996} presented general sufficient
conditions for the almost sure $L_1$-consistency based on
data-dependent partitions. Later in \cite{barron1999}, the authors constructed
a multivariate histogram which achieves asymptotic minimax rates
over anisotropic H\"older classes for the $L_2$ loss. Another
closely-related type of methods is multivariate density estimation
based on wavelet expansions (\cite{donoho1996} and
\cite{tribouley1995}). Along this line, in \cite{neumann2000} and \cite{klemela2009} the authors showed that estimators based on
wavelet expansions achieve minimax convergence rates up to a logarithmic
factor over a large scale of anisotropic Besov classes. Apart from these
two types of methods,
in a recent work \cite{wong2010} by Wong and Ma, the authors proposed a Bayesian formulation to learn the data-adaptive
partition in multi-dimensional cases. By employing
sequential importance
sampling (\cite{KLW} and \cite{liu2001monte}), they designed efficient
algorithms (\cite{LJW} and \cite{jiang2015}) to sample
from the posterior distribution. The methods are also shown to perform well empirically
in a range of
continuous and discrete problems, and achieves satisfying performance.

\subsection{Contribution of the paper}
In this paper we study the asymptotic properties of density estimators
based on adaptive partitioning. The data-adaptive partition is
obtained by maximizing the likelihood of the corresponding histogram
on that partition. We start by formulating a complexity
index (denoted by $r$) for a density, with large value of $r$
indicating low complexity in the sense that the density can be
approximated at a fast rate by piecewise constant density functions as the
size of the underlying partition increases. We assume that the
complexity of the true density is known, and study how the size of the
partition of our density estimator should be chosen in order to
achieve fast convergence to the true density. Our analysis shows that
roughly (i.e. up to $\log n$ factors) the achievable rate is
$n^{-(r/(2r+1))}$. Thus when $r$ is large, our estimate will converge to the
true density at a rate close to the parametric rate of $n^{-1/2}$,
not directly depending on the dimension $d$ of the sample space. This is in
contrast with the achievable convergence rate under smoothness
condition (\cite{stone1980}), which deteriorates rapidly as the
dimensional $d$ increases.

In order to gain a deeper understanding of our complexity index, in
this paper we also perform explicit computation of $r$ for density functions
with certain spatial sparsity properties as well as functions of bounded
variation. Our results also imply that, when the true density depends
only on a subset of the variables, our density estimator will
automatically exhibit a ``variable selection'' property. Indeed, the function classes studied in this paper correspond to Besov
classes and anisotropic H\"older classes studied by the previous
literature (\cite{barron1999},  \cite{neumann2000} and
\cite{klemela2009}). But the density functions are characterized by
the conditions which are easier to verify  and closer
to statistical models. In addition to this, in most of the
previous literature regarding convergence rates, the estimator is constructed for a specific
function space. Here, we
introduce the estimator under a quite general formulation, and then derive
the convergence rates by calculating complexity indices $r$ for different function classes. While the results in this paper can provide insights given the
complexity index $r$ of the true density, in practice we do not know
$r$. This raises the question of how to make our density estimator
adaptive to the unknown complexity, in the sense that it will
automatically achieve the optimal rate $n^{-(r/(2r+1))}$ even when $r$ is not
known to us. This question will be taken up in a companion paper (\cite{linxiliu2015}). Extending the present
analysis, we will show in the companion paper that adaptation to unknown complexity can be achieved by a fully Bayesian approach with an exponentially decreasing prior distribution on the size of the partition.

The rest of the paper is organized in the following way. In Section
\ref{sec:notation}, we discuss the partition scheme, introduce
the estimation method, and summarize our
main results on convergence rate. Section
\ref{sec:rate}  focuses on the proof of the main theorem. If the reader is not interested in the mathematical details, this section
can be skipped without affecting the understanding of the later
part. From Section
\ref{sec:spatialadaptation} to Section \ref{sec:variableselection},
we apply our main results to spatial adaptation, estimation of
functions of bounded variation, and variable selection cases respectively.
\section{The maximum likelihood estimators based on adaptive
  partitioning}
\label{sec:notation}
Let $Y_1, Y_2, \cdots, Y_n$ be a sequence of independent random
variables distributed according to a density $f_0(y)$ with respect to a
$\sigma$-finite measure $\mu$ on a measurable space $(\Omega,
\mathcal{B})$. We are interested in the case when $\Omega$ is a
bounded rectangle in $\mathbb{R}^p$ and
$\mu$ is the Lebesgue measure. After translation and scaling, we may assume that the
sample space is the unit cube in $\mathbb{R}^p$,
that is, $\Omega=\{(y^1, y^2, \cdots, y^p): y^l
\in [0,1]\}$. Let $\mathcal{F}=\{f \mbox{ is a nonnegative measurable
  function on } \Omega: \int_\Omega
f d\mu =1 \}$ be the collection of all the density functions on $(\Omega,
\mathcal{B}, \mu)$. $\mathcal{F}$ constitutes the parameter space in this problem.
\subsection{Densities on binary partitions}
\label{subsec:sieve}
Similar to \cite{LJW}, we use \emph{binary partitions} to capture the
features of the true density function. By increasing the complexity of the partitions, we construct a
sequence of density spaces $\Theta_1, \Theta_2, \cdots,
\Theta_I, \cdots$. These spaces are finite dimensional approximations
to the infinite dimensional parameter space $\mathcal{F}$ with
decreasing approximation error. A detailed description of these spaces
is as
follows.

First, we use a recursive procedure to define a binary partition
with $I$ subregions of the unit cube in
$\mathbb{R}^p$. Let $\Omega=\{(y^1, y^2, \cdots, y^p): y^l
\in [0,1]\}$ be that unit cube. In the first step, we choose
one of the coordinates $y^l$ and cut $\Omega$ into two subregions along the midpoint of
the range of $y^l$. That is, $\Omega=\Omega^l_0 \cup \Omega^l_1$, where $\Omega^l_0 = \{y \in \Omega:
y^l \leq 1/2\}$ and $\Omega^l_1 = \Omega \backslash \Omega^l_0$. In this way, we obtain a
partition with two subregions. It is easily observed that the total number of all possible partitions
after one cut is equal to the dimension $p$. Suppose after $I-1$ steps of the
recursion, we already obtained a partition
$\{ \Omega_i \}_{i=1}^I$ with $I$ subregions. In the $I$-th step, further
partitioning of the region is defined as follows: 
\begin{enumerate}
\item Choose a region from $\Omega_1, \cdots, \Omega_I$. Denote it as
  $\Omega_{i_0}$. 
\item Choose one coodinate $y^l$ and divide $\Omega_{i_0}$ into two subregions
  along the midpoint of the range of $y^l$.
\end{enumerate}
Such a partition obtained by $I$ recursive steps is called
a binary partition of size $I+1$. Figure~\ref{fig:binarypartition}
displays all the two dimensional binary partitions when $I$ is 1, 2 and 3.
\begin{figure} 
\label{fig:binarypartition}
\centering
\includegraphics[scale=0.4]{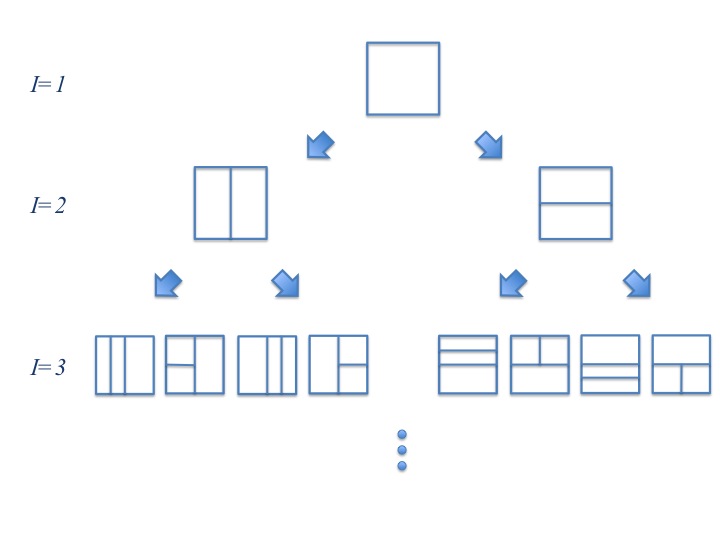}
\caption{Binary partitions}
\end{figure}

Now, we define
\begin{eqnarray*}
\Theta_I = \{f \in \mathcal{F}: f=\sum_{i=1}^I \beta_i
\mathbbm{1}_{\Omega_i},\ \mbox{where\ }\sum_{i=1}^I \beta_i \mu(\Omega_i)
=1,\\\mbox{and\ } \{\Omega_i\}_{i=1}^I\mbox{ is\ a\ binary\ partition\ of\ $\Omega$\ of\ size\ $I$.}\}
\end{eqnarray*}
This is to say, $\Theta_I$ is the collection of piecewise constant density functions
supported by the binary partitions of size $I$. Then these $\Theta_I$
constitute a sequence of approximating spaces to $\mathcal{F}$
(i.e. a sieve, see \cite{Grenander} and \cite{shen1994} for background on sieve theory).

We take the metric on $\mathcal{F}$ and $\Theta_I$ to be Hellinger distance, which is defined by
\begin{equation}
\label{al:hellinger}
\rho(f,g)=\left(\int_{\Omega} \left(\sqrt{f(y)} - \sqrt{g(y)} \right)^2
dy \right)^{1/2},\ f,g \in \mathcal{F}.
\end{equation}
For $f, g \in \Theta_I$, let $f= \sum_{i=1}^I \beta_i^1 \mathbbm{1}_{\Omega_i^1}$ and
$g=\sum_{i=1}^I \beta_i^2 \mathbbm{1}_{\Omega_i^2}$, where $\{
\Omega_i^1\}_{i=1}^I$ and $\{\Omega_i^2\}_{i=1}^I$ 
are binary partitions of $\Omega$. Then the Hellinger distance
between $f$ and $g$ can be written more explicitly as
\begin{equation}
\rho^2(f, g)=\sum_{i=1}^I \sum_{j=1}^I \left(\sqrt{\beta_i^1} -
\sqrt{\beta_j^2}\right)^2 \mu \left(\Omega_i^1 \cap \Omega_j^2 \right).
\end{equation}

We also introduce the Kullback-Leibler divergence, which is
defined to be 
\begin{equation*}
  K(f,g) = \mathbb{E}_{f} \left( \log \frac{f(Y)}{g(Y)} \right).
\end{equation*}
Note the Kullback-Leibler divergence is a stronger distance compared
to the Hellinger distance, in the sense that for any $f, g \in \mathcal{F}$, $
\rho^2(f, g) \leq K(f,g)$. 

We further restrict our interest to a subsect $\mathcal{F}_0 \subset \mathcal{F}$
of densities which satisfy the following two conditions: First, $\int_{\Omega} f^2 <
\infty$. Second, for any $f \in
\mathcal{F}_0$, there exists a sequence of approximations $f_I\in \Theta_I$ such that
$\rho(f, f_I ) \leq A I^{-r}$, where $r$ is a parameter
characterizing the decay rate of the approximation error, $A$ is a
constant that may depend on $f$.
In order to demonstrate that $\mathcal{F}_0$ is still a rich class, from Section
\ref{sec:spatialadaptation} to Section \ref{sec:variableselection}, we
study several
specific density classes belonging to $\mathcal{F}_0$, which are frequently occurred in statistical
modelding.
\subsection{The sieve MLE}
For any $f \in \Theta_I$, the log-likelihood is defined to be
\begin{equation}
  \label{eq:loglikeli}
  L_n(f) = \sum_{j=1}^n \log f(Y_j) = 
  \sum_{i=1}^I N_i \log \beta_i,
\end{equation}
where $N_i$ is the count of data points in $\Omega_i$, i.e., $N_i= \mbox{card}\{j: Y_j \in \Omega_i, 1\leq j\leq n\}$. The maximum likelihood
estimator on $\Theta_I$ is defined to
be
\begin{equation}
  \label{eq:sievemle}
  \hat{f}_{n,I} = \arg\max_{f \in \Theta_I} L_n(f).
\end{equation}
We claim that $\hat{f}_{n,I}$ is well defined. This is true because given the binary partition
$\{ \Omega_i\}_{i=1}^I$, the underlying distribution becomes a multinomial
one, and $(\beta_1, \cdots, \beta_I)$ can be determined
by maximizing the log-likelihood. And within each $\Theta_I$, the number of possible binary partitions
is finite. Since $\Theta_I$ constitute a sieve to $\mathcal{F}_0$, this
sequence of estimators is also called the sieve maximum likelihood
estimator (sieve MLE).

In order to illustrating how the idea of partition learning is incorporated in this framework, given the
binary partition $\mathcal{A} =\{ \Omega_i \}_{i=1}^I$, we can derive
the maximum of the likelihood values achieved by histograms on that partition, which has a closed-form expression
\begin{equation*}
  L_n (\mathcal{A}) = \sum_{i=1}^I N_i \log \left( \frac{N_i}{n
      \mu(\Omega_i)} \right). 
\end{equation*}
We treat this as a \emph{score} of the partition $\mathcal{A}$. By
maximizing the score over all the binary partitions of a fixed size, we
learn a most promising one which adapts to the true density
function. Then $\hat{f}_{n,I}$ is simply the histogram based on that partition.

\subsection{Main results on convergence rate}
Having defined a sequence of maximum likelihood
estimators $\hat{f}_{n,I}$, we are now ready to state the main
results on the rate at which $\hat{f}_{n,I}$
converges to $f_0$.

\begin{theorem}
\label{th:convergencerate}
For any $f_0 \in \mathcal{F}_0$, $\hat{f}_{n,I}$ is the maximum
likelihood estimator over $\Theta_I$. $A$ and $r$ are the
parameters that characterize achievable approximation errors to $f_0$ by the
elements in $\Theta_I$. Assume that
$n$ and $I$ satisfy
\begin{equation}
  \label{eq:thmC2}
  I = \left( (2^8 A^2 r /c_1) \frac{n}{\log n} \right)^{\frac{1}{2r+1}},
\end{equation}
where the constant $c_1$ can be chosen to be in $(0, 1)$. Then the convergence rate of the
sieve MLE is $ n^{-\frac{r}{2r+1}} (\log n)
^{(\frac{1}{2}+\frac{r}{2r+1})}$, in the sense that
\begin{equation*}
 \mathbb{P}_{f_0} \left( \rho (\hat{f}_{n,I} , f_0 ) \geq D n^{-\frac{r}{2r+1}} (\log n)
^{(\frac{1}{2}+\frac{r}{2r+1})} \right) \rightarrow 0,  
\end{equation*}
where $D$ is a constant.  
\end{theorem}

\begin{remark}
  It is possible to partition the sample space in a more flexible way. In
  particular, the analysis and resulting rate remain the same if we replace binary
  partition at the mid-point by binary partition at a point chosen from a
  fixed sized grid (e.g. regular equi-spaced grid).
\end{remark}

From this theorem, we see that the convergence rate does not
directly depend on the dimension of the problem. Instead, it only
depends on how well the true density can be approximated by the
sieve. As $r$ increases, up to a $\log n$ term, the rate gets close to
the parametric rate of $n^{-1/2}$. 

The analysis in this paper will demonstrate that the optimal convergence rate
can be achieved by balancing the sample size with the complexity of the
approximating spaces. On one hand, the complexity of $\Theta_I$ affects the
convergence rate in a way that, the richer the approximating spaces
the lower bias the estimators have. Conversely, given a sample $Y_1, Y_2, \cdots, Y_n$ of fixed size, there is a point beyond
which the limited amount of information conveyed in the
data may be overwhelmed by the overly-complex approximating spaces. A
major contribution of the main result is that it clarifies how to strike
the balance between the sample size and the complexity of the
approximating spaces.
\section{Proof of Theorem \ref{th:convergencerate}}
\label{sec:rate}
This section is devoted to the proof of the main theorem. Studies of
convergence rate invariably rely on the previous
results from studies of empirical process indexed by
log-likelihood ratios. However, while the results in the classical
work by Wong and Shen (\cite{shen1994}
and \cite{wong1995}) are the most applicable ones to our study, they
must be modified to adapt to the current settings. The following
is an outline of the proof.

In Section \ref{subsec:entropy} we briefly
discuss {\em metric entropy with bracketing}, which measures the
complexity of the approximating spaces by
``counting'' how many pairs of functions in an $\epsilon$-net are
needed to provide simultaneous upper and lower bounds of all the elements. An
important result of this section is an upper bound for the bracketing
metric entropy of $\Theta_I$. In Section
\ref{subsec:likelihoodratio}, a large-deviation type inequality for
the likelihood ratio surface follows.

Previous results on the convergence rate of sieve MLE
assume that the true parameter can be approximated by the sieve under
Kullback-Leibler divergence. Here, we switch to the weaker
Hellinger distance, because under the current settings, we can obtain
an explicit bound for the Kullback-Leibler divergence in terms of the
Hellinger distance. Thus our results are more general than those in
\cite{wong1995} for this type of density estimation problems.

Finally in Section
\ref{subsec:rate}, we establish the main result of the
convergence rate. After splitting the tail probability into two parts
corresponding to variance and bias,
we apply the results obtained in Section \ref{subsec:likelihoodratio}
and \ref{subsec:KLbound} to bound each of them respectively.

\subsection{Calculation of the metric entropy with bracketing}
\label{subsec:entropy}
A general discussion of {\em metric entropy} can be found in
\cite{kolmogorov1992selected}.
In this section, we introduce a form of metric
entropy with bracketing corresponding to current parameter space, 
and provide an upper bound for the bracketing metric entropy of the
approximating spaces defined in Section \ref{subsec:sieve}. 
\begin{definition}
  Let $ (\Theta, \rho)$ be a seperable pseudo-metric
  space. $\Theta(\epsilon)$ is a finite set of pairs of functions $\{
  (f_j^L, f_j^U), j=1,\cdots, N\}$ satisfying
  \begin{equation}
    \label{eq:entropyDefC1}
    \rho( f_j^L, f_j^U) \leq \epsilon\ for\ j=1,\cdots,N,
  \end{equation}
and for any $f \in \Theta$, there is a $j$ such that
\begin{equation}
  \label{eq:entropyDefC2}
  f_j^L \leq f \leq f_j^U.
\end{equation}
Let 
\begin{equation}
  \label{eq:entropyDefN}
  N(\epsilon, \Theta, \rho) = min \{card\ \Theta(\epsilon): (\ref{eq:entropyDefC1})\ and\ (\ref{eq:entropyDefC2})\ are\ satisfied
\}.
\end{equation}
Then, we define the metric entropy with bracketing of $\Theta$ to be 
\begin{equation}
  \label{eq:entropyDef}
  H(\epsilon, \Theta, \rho) = \log N(\epsilon, \Theta, \rho).
\end{equation}
\end{definition}

Recall that $\Theta_1, \cdots, \Theta_I, \cdots$ are the approximating spaces defined in section
\ref{subsec:sieve}. The
next two lemmas are devoted to an upper bound for the bracketing metric entropy of
$\Theta_I$.
\begin{lemma}
  \label{lemma:boundofH}
Take $\rho$ to be the Hellinger distance. Let $\Theta_I^{\mathcal{A},
  d} = \{ f \in \Theta_I:\ f\ is\ supported\ by\ the\ binary\
partition\ \mathcal{A} = \{ \Omega_i\}_{i=1}^I,\ and\ \rho(f, f_0) \leq d \}$. Then, 
\begin{equation}
  \label{eq:boundofH}
  H(u, \Theta_I^{\mathcal{A}, d}, \rho) \leq \frac{I}{2} \log I + I
  \log\frac{d}{u} + b', 
\end{equation}
where $b'$ is a constant not dependent on the binary partition.
\end{lemma}
\begin{proof}
 Assume $f=\sum_{i=1}^I
  \beta_i \mathbbm{1}_{\Omega_i}$. When the binary partitions $\{\Omega_i\}_{i=1}^I$ are fixed, there exits a one-to-one
  correspondance between any $f \in \Theta_I^{\mathcal{A}, d}$ and an $I$-dimensional
  vector $(\sqrt{\beta_1 \mu(\Omega_1)}, \cdots, \sqrt{\beta_I
    \mu(\Omega_I)})$.
As a consequence of Cauchy-Schwartz inequality,
  \begin{eqnarray*}
\rho(f, f_0)^2&=&\sum_{i=1}^I \int_{\Omega_i} ( \sqrt{\beta_i}
    -\sqrt{f_0(x)} )^2 \mu(dx)\\
&\geq&\sum_{i=1}^I \mu(\Omega_i) (\int_{\Omega_i}
(\sqrt{\beta_i}-\sqrt{f_0(x)}) \frac{\mu(dx)}{\mu(\Omega_i)})^2\\
&=&\sum_{i=1}^I \mu(\Omega_i) (\sqrt{\beta_i} - \frac{\int_{\Omega_i}
  \sqrt{f_0(x)} \mu(dx)} {\mu(\Omega_i)})^2. 
  \end{eqnarray*}
Then, we have,
\begin{eqnarray*}
& & \{(\sqrt{\beta_1 \mu(\Omega_1)}, \cdots, \sqrt{\beta_I \mu(\Omega_I)}): \sum_{i=1}^I \int_{\Omega_i} ( \sqrt{\beta_i}
    -\sqrt{f_0(x)} )^2 \mu(dx) \leq d^2\}\\
&\subset& \{(\sqrt{\beta_1 \mu(\Omega_1)}, \cdots, \sqrt{\beta_I \mu(\Omega_I)}):
\sum_{i=1}^I (\sqrt{\beta_i \mu(\Omega_i)} - \frac{\int_{\Omega_i}
  \sqrt{f_0(x)} \mu(dx)} {\sqrt{\mu(\Omega_i)}}  )^2 \leq d^2 \} \\
&=:& B_I ^{\mathcal{A}, d}.
\end{eqnarray*} 
If we treat the element in $\Theta_I^{\mathcal{A},d}$ as the
$I$-dimensional vector $( \sqrt{\beta_1 \mu(\Omega_1)}, \cdots, \sqrt{\beta_I
    \mu(\Omega_I )} )$, then from the above inclusion relation we learn that,
  $\Theta_I^{\mathcal{A},d} \subset B_I^{\mathcal{A},d}$. We also note
  that the
  Hellinger distance on  $B_I^{\mathcal{A},d}$ is equivalent to the
  $L_2$ norm on the $I$-dimensional Euclidean space. Thus,
\begin{equation}
  \label{eq:hellingertol2}
  N(u, \Theta_I^{\mathcal{A}, d}, \rho) \leq N(u, B_I^{\mathcal{A}, d},\|
  \cdot \|_2).
\end{equation}
Because the metric entropy is invariant under translation, calculating the bracketing
metric entropy of $B_I^{\mathcal{A}, d}$ is equivalent to calculating
that of
\begin{equation*}
\tilde{B}_I^{\mathcal{A}, d} := \{ \left(\sqrt{\beta_1 \mu(\Omega_1)}, \cdots,
\sqrt{\beta_I \mu(\Omega_I)} \right): \sum_{i=1}^I \left(
\sqrt{\beta_i \mu(\Omega_i)} \right)^2 \leq d^2 \}.
\end{equation*}
The unit sphere under $L_2$-norm is
\begin{equation*}
  S=\{ \left(\sqrt{\beta_1 \mu(\Omega_1)}, \cdots, \sqrt{\beta_I
      \mu(\Omega_I)} \right): \sum_{i=1}^I \left(\beta_i \mu(\Omega_i)
    \right) \leq 1 \}.
\end{equation*}
The unit sphere under $L_{\infty}$-norm is 
\begin{equation*}
  S_{\infty} = \{ \left(\sqrt{\beta_1 \mu(\Omega_1)}, \cdots,
    \sqrt{\beta_I \mu(\Omega_I)} \right): \max_{1 \leq i
    \leq I} \sqrt{\beta_i \mu(\Omega_i)} \leq 1 \}.
\end{equation*}
Note that  $ \max_{1 \leq i \leq I} \sqrt{\beta_i \mu(\Omega_i)} \leq 1/
\sqrt{I}$ implies that $\sum_{i=1}^I \beta_i \mu(\Omega_i) \leq 1$, and $\sum_{i=1}^I \beta_i \mu(\Omega_i) \leq d^2$ implies that
$\max_{1 \leq i \leq I}\sqrt{\beta_i \mu(\Omega_i)} \leq d$, we have 
\begin{equation}
  \label{eq:relation}
 S \subset \tilde{B}_I^{\mathcal{A}, d}
  \subset d S_{\infty} \mbox{\ and\ } S_{\infty} \subset \sqrt{I} S. 
\end{equation}
Therefore, 
\begin{eqnarray*}
  N(u, \Theta_I ^{\mathcal{A}, d}, \rho) &\leq& N(u, \tilde{B}_I^{\mathcal{A}, d}, \| \cdot \|_2) \\
&\leq& \left(\frac{d \sqrt{I}}{u} +2 \right)^I \\
&\leq& b' I^{I/2} (d/u)^I,
\end{eqnarray*}
where $b'$ is a constant not dependent on the partiton. The
desired result follows.
\end{proof}

\begin{lemma}
  \label{lemma:dballupperbound}
Under the same assumptions as in Lemma \ref{lemma:boundofH}, let $\Theta_I^d = \{ f \in \Theta_I: \rho(f, f_0) \leq
d\}$. Then, 
\begin{eqnarray}
  \label{eq:dballupperbound}
\nonumber  & &H(u, \Theta_I^d, \rho) \\
&\leq& I \log p + (I+1) \log(I+1) + \frac{I}{2} \log I + I \log \frac{d}{u} +b,
\end{eqnarray}
where $b$ is a constant not dependent on $I$ or $d$. 
\end{lemma}
\begin{proof}
According to
the construction of the sieve, given the size $I$, the
number of possible binary partitions is upper bounded by $p^I I!$ ($p$
is the dimension of the Euclidean space). Therefore, 
\begin{eqnarray*}
  \nonumber N(u, \Theta_I^d, \rho) &\leq& p^I I! N(u,
  \Theta_I^{\mathcal{A}, d}, \rho) \\
&\leq& b' p^I I! I^{I/2} (d/u)^I,
\end{eqnarray*}
and,
\begin{eqnarray}
  \label{eq:upperE5}
\nonumber  & &\ H(u, \Theta_I^d,\rho) \\
&\leq& I \log p + (I+1) \log(I+1) + \frac{I}{2} \log I + I \log \frac{d}{u} +b.
\end{eqnarray}
\end{proof}

\subsection{An inequality for the likelihood ratio surface}
\label{subsec:likelihoodratio}
In this section, we focus on bounding the tail behavior of the likelihood
ratio. First, we cite a theorem in \cite{wong1995},
which gives a uniform exponential bound for likelihood ratios. 

\begin{theorem}[Wong and Shen (1995)]
\label{thm:likelihoodratio}
  Let $\rho$ be the Hellinger distance and $\mathcal{P}_n$ be a space of
  densities. There exist positive constants $a>0$, $c$, $c_1$ and $c_2$, such that,
  for any $\epsilon>0$, if 
  \begin{equation}
    \label{eq:entropycondition}
    \int_{\epsilon^2/2^8}^{\sqrt{2}\epsilon}H^{1/2} ( u/a, \mathcal{P}_n,
    \rho) du \leq c n^{1/2} \epsilon^2,
  \end{equation}
then
\begin{equation*}
  \mathbb{P}_{f_0} \Big( \sup_{\{ \rho(f, f_0) \geq \epsilon, f \in
    \mathcal{P}_n \}} \prod_{i=1}^n \frac{f(Y_i)}{f_0(Y_i)} \geq \exp( -c_1
  n \epsilon^2) \Big) \leq 4 \exp(-c_2 n \epsilon^2),
\end{equation*}
where $\mathbb{P}_{f_0}$ is understood to be the outer probability
mesure under $f_0$. The constants $c_1$ and $c_2$ can be chosen in
$(0,1)$ and $c$ can be set as $(2/3)^{5/2} /512$.
\end{theorem}

Combining the theorem with the entropy bounds in Section
\ref{subsec:entropy} gives the desired inequality, which is summarized
in the next corollary.

\begin{corollary}
\label{cor:largedeviation}
  Let $\delta_{n,I} = (\frac{I \log I}{n / \log n}) ^{1/2}$. When $n$ and $I$ are sufficiently
  large, we have 
  \begin{equation*}
    \mathbb{P}_{f_0} \Big( \sup_{\{ \rho(f, f_0) \geq \delta_{n,I}, f \in
   \Theta_I \}} \prod_{i=1}^n \frac{f(Y_i)}{f_0(Y_i)} \geq \exp( -c_1
  n \delta_{n,I}^2) \Big) \leq 4 \exp(-c_2 n \delta_{n,I}^2).
  \end{equation*}
\end{corollary}
\begin{proof}
  The key to the proof is to check condition (\ref{eq:entropycondition}) so
  that we can apply Theorem \ref{thm:likelihoodratio} to the sequence
  of approximating spaces $\Theta_I$. 
By Lemma \ref{lemma:dballupperbound},
\begin{eqnarray}
\label{eq:entropybound}
\nonumber  & &\int_{\delta_{n,I}^2/2^8}^{\sqrt{2} \delta_{n,I}} H^{1/2} ( u/a,
  \Theta_I, \rho) du \\
\nonumber &\leq& \int_{\delta_{n,I}^2/2^8}^{\sqrt{2} \delta_{n,I}}
\left( I \log (4pa) + 2 (I+1) \log(I+1) -I \log u \right)^{1/2} du \\
&\sim& I^{1/2} \int_{\delta_{n,I}^2/2^8}^{\sqrt{2} \delta_{n,I}} \left( \log\frac{I^2}{u}
\right)^{1/2} du.
\end{eqnarray}
We calculate the integral and obtain
\begin{eqnarray*}
& & \eqref{eq:entropybound}\\
&\leq& I^{1/2} \left( u \sqrt{ \log \frac{I^2}{u} } - \frac{\sqrt{\pi}}{2}
I^2 \mbox{erf} \left( \sqrt{ \log \frac{I^2}{u} } \right) \right)|_{\delta_{n,I}^2/2^8}^{\sqrt{2} \delta_{n,I}} \\
&\leq& I^{1/2} ( \sqrt{2} \delta_{n,I} \sqrt{ \log \frac{I^2}{
    \sqrt{2} \delta_{n,I}} } - \frac{\sqrt{\pi}}{2}
I^2 \mbox{erf} \left( \sqrt{ \log \frac{I^2}{ \sqrt{2}\delta_{n,I} } } \right) -\frac{\delta_{n,I}^2}{2^8} \sqrt{ \log
  \frac{2^8 I^2}{ \delta_{n,I}^2} }\\
& & + \frac{\sqrt{\pi}}{2}
I^2 \mbox{erf} \left( \sqrt{ \log \frac{2^8 I^2}{\delta_{n,I}^2} } \right) ) \\
&\sim& I^{1/2} \left(  \delta_{n,I} \sqrt{ \log \frac{I^2}{\delta_{n,I}} } + \frac{\sqrt{\pi}}{2}
I^2  \left(\mbox{erf} ( \sqrt{ \log \frac{2^8 I^2}{\delta_{n,I}^2} }
 ) -\mbox{erf} (
\sqrt{ \log \frac{I^2}{\sqrt{2} \delta_{n,I}} }) \right) \right) \\
&\sim& I^{1/2} \left(  \delta_{n,I} \sqrt{ \log \frac{I^2}{\delta_{n,I}} } + \frac{\sqrt{\pi}}{2}
I^2 \left( 1 - \frac{ \frac{ \delta_{n,I}^2 }{ 2^8 I^2}} {\sqrt{ \pi \log
    \frac{2^8 I^2}{\delta_{n,I}^2}}} -1 + \frac{ \frac{\sqrt{2}
    \delta_{n,I} }{I^2} } {
  \sqrt{\pi \log \frac{I^2}{\sqrt{2} \delta_{n,I} }} } \right) \right) \\
&\sim& I^{1/2} \delta_{n,I} \sqrt{ \log \frac{I^2}{\delta_{n,I}} } \\
&\leq & c \sqrt{n} \delta_{n,I}^2.
\end{eqnarray*}
Therefore, condition (\ref{eq:entropycondition}) is satisfied. The
desired result follows from Theorem \ref{thm:likelihoodratio}.
\end{proof}
\subsection{An inequality for the Kullback-Leibler divergence}
\label{subsec:KLbound}

It is well known that Hellinger distance can be bounded by Kullback-Leibler divergence. In \cite{wong1995}, the authors showed that
the other direction also holds under mild conditions. This type of
result becomes quite useful in this paper because it would allow us to
study the convergence rate under the weaker Hellinger distance. We
first cite the result from their paper. It enables us to obtain a more
explicite bound for density functions in $\mathcal{F}_0$ in the later part of this section. 

\begin{lemma} [Wong and Shen (1995) Theorem 5]
  \label{lemma:hellingertokl}
  Let $f$, $f_0$ be two densities, $\rho^2(f_0, f) \leq
  \epsilon^2$. Suppose that $M_{\lambda} ^2 = \int_{\{f_0/f \geq
    e^{1/\lambda} \} } f_0 (f_0 /f)^{\lambda} < \infty$ for some $\lambda
  \in (0, 1]$. Then for all $\epsilon^2 \leq \frac{1}{2} (
  1-e^{-1})^2$, we have 
  \begin{eqnarray*}
    \int f_0 \log (\frac{f_0}{f}) \leq \big[ 6 + \frac{2 \log
      2}{(1-e^{-1})^2} + \frac{8}{\lambda} \max \big( 1 , \log
    (\frac{M_{\lambda}}{\epsilon}) \big) \big] \epsilon^2,\\
\int f_0 \big( \log (\frac{f_0}{f}) \big)^2 \leq 5\epsilon^2 \big[
\frac{1}{\lambda} \max \big( 1, \log (\frac{M_{\lambda}}{\epsilon} )
\big) \big]^2.
  \end{eqnarray*}
\end{lemma}

When $f_0$ is the true density function and $f$ is an approximation to
$f_0$ in $\Theta_I$, we can obtain a more explicit bound of the
Kullback-Leibler divergence. The result is summarized in the lemma
below.

\begin{lemma}
\label{lemma:KLbound}
  $f_0$ is a density function on $\Omega$. If $f_0 \in \mathcal{F}_0$, then we can find
  $g \in \Theta_I$, such that 
  \begin{eqnarray*}
   \int f_0 \log (f_0/g)  \leq 128 A^2 r I^{-2r} \log I,\\
\int f_0 (\log ( f_0/g) )^2 \leq 320 A^2 r^2 I^{-2r} (\log I)^2.
  \end{eqnarray*} 
\end{lemma}
\begin{proof}
  Assume that $f= \sum_{i=1}^I \beta_i \mathbbm{1}_{\Omega_i}$, where
  $\{ \Omega_i \}_{i=1}^I$ is a binary partition of $\Omega$. From the
  property of $L_2$-projection, we have
  \begin{eqnarray*}
    \rho^2(f_0, f) &=& \sum_{i=1}^I \int_{\Omega_i} ( \sqrt{f_0(x)} -
    \sqrt{\beta_i})^2 \mu(dx) \\
&\geq & \sum_{i=1}^I \int_{\Omega_i} \left( \sqrt{f_0} -
\frac{\int_{\Omega_i} \sqrt{f_0} } { \mu (\Omega_i)} \right)^2.
  \end{eqnarray*}
Let $h=\sum_{i=1}^I \left( \frac{\int_{\Omega_i}
    \sqrt{f_0}}{\mu(\Omega_i)} \right)^2 \mathbbm{1}_{\Omega_i}$, then
\begin{equation*}
  \int_{\Omega} h = \sum_{i=1}^I \frac{ (\int_{\Omega_i} \sqrt{f_0} )^2
} {\mu(\Omega_i)} \leq \sum_{i=1}^I \frac{( \int_{\Omega_i} f_0
  )\mu(\Omega_i)}{\mu(\Omega_i)} =1.
\end{equation*}

Let $g= \frac{h} {\int_{\Omega} h}$, then $g$ is a density function,
and 
\begin{eqnarray*}
  \rho^2 (f_0, g) &=& \| \sqrt{f_0} -\sqrt{h} + \sqrt{h} -\frac{
    \sqrt{h}}{ \| \sqrt{h} \|_2 } \|_2^2\\
&\leq& 2 \| \sqrt{f_0} -\sqrt{h} \|_2^2 + 2 (1-\frac{1}{\| \sqrt{h}
  \|_2} )^2 \| \sqrt{h} \|_2^2 \\
&\leq & 2\rho^2(f_0, f) + 2(1- \| \sqrt{h} \|_2^2) \\
&=& 2\rho^2(f_0,f) + 2 \| \sqrt{f_0} -\sqrt{h} \|_2^2 \\
&\leq& 4 \rho^2(f_0,f).
\end{eqnarray*}
For density functions $f_0$ and $g$,
\begin{eqnarray*}
  M_{1/4}^2 &=& \sum_{i=1}^I \int_{\Omega_i \cap \{ f_0 / g >
    e^4 \} } f_0 \left( \frac{f_0}{ \left( \frac{\int_{\Omega_i}
        \sqrt{f_0} } { \| \sqrt{h} \|_2 \mu(\Omega_i)} \right) ^2}
\right)^{1/4} \\
&\leq& \sum_{i=1}^I \frac{ \int_{\Omega_i} f_0^{1+1/4}} { \left(
    \frac{\int_{\Omega_i} \sqrt{f_0} }{\mu(\Omega_i)} \right)^{1/2}}
\\
&\leq& \sum_{i=1}^I \frac{ (\int_{\Omega_i} f_0 ^{1/2})^{1/2}
  (\int_{\Omega_i} f_0 ^2 )^{1/2} }{ \left(
    \frac{\int_{\Omega_i} \sqrt{f_0} }{\mu(\Omega_i)} \right)^{1/2}}= \sum_{i=1}^I (\int_{\Omega_i} f_0^2 ) ^{1/2} (\mu(\Omega_i))^{1/2}
\\
&\leq& \left(\int_{\Omega} f_0 ^2 \right) ^{1/2}.
\end{eqnarray*}
Therefore, if we set $\lambda=1/4$, when $I$ is large enough
\begin{eqnarray*}
  \int f_0 \log (f_0/g) &\leq& [6 + \frac{2 \log2}{(1-e^{-1})^2}
  +32 \max (1, \log \frac{ (\int f_0^2)^{1/4}}{ 2 A I^{-r} }) ] \cdot
  4 A^2 I^{-2r} \\
&\leq& 128 A^2 r I^{-2r} \log I.
\end{eqnarray*}
Similarly, 
\begin{equation*}
  \int f_0 (\log ( f_0/g) )^2 \leq 320 A^2 r^2 I^{-2r} (\log I)^2.
\end{equation*}
\end{proof}

\subsection{Convergence rate}
\label{subsec:rate}
In this section, we apply the above large-deviation type inequality and the bound of the
Kullback-Leibler divergence to derive convergence rate of the sieve MLE for $f_0 \in \mathcal{F}_0$.

\begin{proof}[proof of Theorem \ref{th:convergencerate}]
For any $g \in \Theta_I$ and $D>1$, we have 
\begin{equation*}
  \mathbb{P} \left( \rho(f_0, \hat{f}_{n,I} ) \geq D \delta_{n,I} \right) \leq
  \mathbb{P}_{f_0} \left( \sup_{ \{ \rho(f_0, f) \geq D \delta_{n,I}, f \in
    \Theta_I \}} \prod_{j=1}^n f(Y_j)/g(Y_j) \geq 1 \right),
\end{equation*}
where $\mathbb{P}_{f_0}$ is understood to be the outer probability measure under
$f_0$. 

Let $C= \{ f \in \Theta_I: \rho(f_0, f) \geq D \delta_{n,I} \}$. Then for
$g \in \Theta_I$,
\begin{equation*}
  \mathbb{P}_{f_0} \left( \sup_{f \in C} \prod_{j=1}^n
  \frac{f(Y_j)}{g(Y_j)} \geq 1 \right) \leq P_1 +P_2,
\end{equation*}
where
\begin{equation*}
  P_1 = \mathbb{P}_{f_0} \left( \sup_{f \in C} \prod_{j=1}^n
  \frac{f(Y_j)}{f_0(Y_j)} \geq \exp\left(-c_1 n (D \delta_{n,I})^2 \right) \right),
\end{equation*}
\begin{equation*}
  P_2 = \mathbb{P} \left( \prod_{j=1}^n \frac{f_0(Y_j)}{g(Y_j)} \geq
    \exp \left(  c_1 n ( D \delta_{n,I})^2 \right) \right).
\end{equation*}
In regards of $P_1$, Corollary \ref{cor:largedeviation} still applies here if we replace
$\delta_{n,I}$ by $D \delta_{n,I}$. Thus, $P_1 \leq 4 \exp( - c_2 n
D^2 \delta^2_{n,I})$. 

To bound $P_2$, we write
\begin{eqnarray*}
  P_2 &=& \mathbb{P} \left( \sum_{j=1}^n \log \left( \frac{f_0}{g}
    \right) (Y_j) \geq c_1 n (D \delta_{n,I})^2 \right) \\
&=& \mathbb{P} \left( \sum_{j=1}^n \left[ \log \left(\frac{f_0}{g}
    \right) (Y_j) - \mathbb{E} \log \left(\frac{f_0}{g}
    \right) \right] \geq c_1 n (D \delta_{n,I} )^2 - n \int f_0
  \log \left( \frac{f_0}{g} \right) \right).
\end{eqnarray*}
If $\int f_0 \log (f_0/g) < c_1 D^2 \delta^2_{n,I}$, then
\begin{equation*}
  P_2 \leq \frac { n \int f_0 \log \left(\frac{f_0}{g} \right)^2 }{
    n^2 \left( c_1 D^2 \delta^2_{n,I} - \int f_0 \log
      \left(\frac{f_0}{g} \right) \right)^2 }.
\end{equation*}
Based on our assumption, there exists $f_I \in \Theta_I$, such that
$\rho(f_0, f_I) \leq A I^{-r}$. Then from Lemma \ref{lemma:KLbound},
we know that
\begin{eqnarray*}
  \int f_0 \log (f_0/f_I)  \leq 128 A^2 r I^{-2r} \log I,\\
\int f_0 (\log ( f_0/f_I) )^2 \leq 320 A^2 r^2 I^{-2r} (\log I)^2.
\end{eqnarray*}
Therefore,
\begin{equation*}
  \inf_{g \in \Theta_I} P_2 \leq \frac{320 A^2 r^2 I^{-2r} (\log I)^2}
  { n \left( c_1 D^2 \delta^2_{n,I} - 128 A^2 r I^{-2r} \log I \right)^2}.
\end{equation*}

If we take $I =\left( (2^8 A^2 r /c_1) \frac{n}{\log n} \right)
^{\frac{1}{2r+1} }$, then the condition $\int f_0 \log (f_0/f_I) < c_1
D^2 \delta^2_{n,I}$ is satisfied. If $n$ and $I$ are matched in this
way, the order of $\delta_{n,I}$
determines the final convergence rate, which is $n^{-\frac{r}{2r+1}} (\log n)
^{(\frac{1}{2}+\frac{r}{2r+1})}$. This finishes the proof.  
\end{proof}

\section{Application to spatial adaptation}
\label{sec:spatialadaptation}
In this section, we assume that the density concentrates
spatially. Mathematically, this implies the density function
satisfies a type of \emph{spatial sparsity}. In the past two decades,
sparsity has become one of the most discussed types of structure under
which we are able to overcome the curse of
dimensionality. A remarkable example is that it allows us to solve high-dimensional linear
models, especially when the system is underdetermined. It would be
interesting to study how we could benefit from the sparse structure
when performing density estimation. This section is devoted to
this purpose. 
Under current settings, it is natural to characterize the spatial sparsity by controlling the decay
rate of the ordered Haar wavelet coefficients. After introducing the
high-dimensional Haar basis in Section \ref{subsec:Haarbasis1}, we
provide a rigorous characterization of the sparsity condition in Section
\ref{subsec:sparsity}, and illustrate this characterization by several examples. In
Section \ref{subsec:spatialadaptationrate}, we demonstrate that the
sparse structure allows fast convergence by calculating the explicit convergence
rate. 

\subsection{High-dimensional Haar basis}
\label{subsec:Haarbasis1}
Haar basis is the simplest but widely used wavelet basis. In one
dimension, the Haar wavelet's mother wavelet function is
\begin{equation*}
  \psi(y) = \begin{cases} 1  &\mbox{if } 0\leq y < 1/2, \\
                                       -1 & \mbox{if } 1/2 \leq y < 1,
                                       \\
                                       0 & \mbox{otherwise.} \end{cases}
\end{equation*}
And its scaling function is
\begin{equation*}
  \phi(y) = \begin{cases} 1 &\mbox{if } 0 \leq y < 1, \\
                                       0 &\mbox{otherwise.} \end{cases}
\end{equation*}
Here, we take the two-dimensional case to
illustrate how the system is built. This construction
can be extended to high dimensional cases as well.

The two-dimensional scaling function is defined to be
\begin{equation*}
  \phi \phi(y^1,y^2) :=\phi(y^1) \phi(y^2),
\end{equation*}
and three wavelet functions are
\begin{equation*}
  \phi \psi(y^1,y^2) := \phi(y^1) \psi(y^2),
\end{equation*}
\begin{equation*}
  \psi \phi(y^1, y^2) := \psi(y^1) \phi(y^2),
\end{equation*}
\begin{equation*}
 \psi \psi(y^1, y^2) := \psi(y^1) \psi(y^2).
\end{equation*}
If we use a superscript $l$ to index the scaling level of the wavelet
function and subscripts $i$ and $j$ ($i$ and $j$ can be 0 or 1) to denote the horizontal and vertical
translations respectively, then the scales and translates of the three wavelet
functions $\phi \psi$, $\psi \phi$ and $\psi \psi$ are defined to be 
\begin{eqnarray*}
  &\phi \psi_{ij}^l (y^1, y^2) := (\sqrt{2} )^{2 \cdot l} \phi \psi(2^l y^1-i, 2^l y^2 -j),\\
 &\psi \phi_{ij}^l (y^1, y^2) :=  (\sqrt{2} )^{2 \cdot l} \psi \phi(2^l y^1 -i, 2^l y^2 -j),\\
& \psi \psi_{ij}^l (y^1, y^2) :=  (\sqrt{2} )^{2 \cdot l} \psi \psi (2^l y^1 -i, 2^l y^2 -j).
\end{eqnarray*}
These functions together with the single scaling
function $\phi \phi$ define the two-dimensional Haar wavelet basis
\textbf{$\Psi$}.

\subsection{Spatial sparsity}
\label{subsec:sparsity}
Let $f$ be a $p$ dimensional density function and \textbf{$\Psi$}
the $p$-dimensional Haar basis constructed as above. We will work with
$g=\sqrt{f}$ first. Note that $g \in L^2([0,1]^p)$. Thus we can expand $g$ with
respect to \textbf{$\Psi$} as $g = \sum_{\psi \in \Psi} <g,\psi> \psi$
(here $\psi$ is a basis function instead of the Haar
wavelet's mother wavelet function defined above). We
rearrange this summation by the size of wavelet coefficients. In other
words, we order the coefficients as the following
\begin{equation*}
 |<g, \psi_{(1)} >| \geq | <g, \psi_{(2)}>| \geq \cdots \geq |<g,
\psi_{(k)}>| \geq \cdots,
\end{equation*}
then the sparsity condition imposed on the density functions is
that the decay of the wavelet coefficients follows a power law,
\begin{equation}
  \label{eq:powlawdecay}
  |<g,\psi_{(k)}> | \leq C k^{-\beta} \mbox{ for all $k \in \mathbb{N}$
    and $\beta > 1/2$, }
\end{equation}
where $C$ is a constant. 

This condition has been widely used to characterize the sparsity of
signals and images (\cite{abramovich2006} and \cite{Candes}). In particular, in \cite{DeVore}, it was
shown that for
two-dimensional cases, when $\beta > 1/2$, this condition reasonably
captures the sparsity of real world images. 

Next we use several examples to illustrate how this condition
implies the spatial sparsity of the density function.
 
\begin{example}
Assume that the we are studying a density
function in three dimensions. All the mass concentrates in a dyadic cube. Without loss of
generality, we assume that $f_0 = 64 \mathbbm{1}_{ \{0\leq y^1, y^2, y^3 < 1/4
  \}}$, so that all the mass is concentrated in a small cubical subregion. In the three-dimensional case, the single scaling function is $\phi \phi
\phi$, and the seven wavelet functions are defined as
\begin{eqnarray*}
  &\chi_1 = \psi \phi \phi,\ \chi_2 =\phi \psi \phi,\ \chi_3 = \phi \phi
  \psi, \\
& \chi_4 = \psi \psi \phi,\ \chi_5 = \psi \phi \psi,\ \chi_6 = \phi \psi
\psi, \\
&\chi_7=\psi \psi \psi.
\end{eqnarray*}
If we still use the superscript to denote the scaling level and the subscripts
to denote the spatial translations, then the expansion of $f_0$ with
respect to the Haar basis is
\begin{equation*}
  f_0 = \phi \phi \phi + \sum_{k=1}^7
  \chi_{k,000}^{(0)} + 2 \sqrt{2} \sum_{k=1}^7 \chi_{k, 000}^{(1)}.
\end{equation*}
The coefficients display a decaying trend, although the number of them
is finite. More generally, any density function whose expansion only
has finite terms will satisfy the condition (\ref{eq:powlawdecay}).
\end{example}

\begin{example}
Assume that the two-dimensional true density function is
\begin{equation*}
  \begin{pmatrix} Y_1 \\ Y_2 \\ \end{pmatrix} \sim \frac{2}{5} \mathcal{N} \left( \begin{pmatrix} 0.25 \\
  0.25 \\ \end{pmatrix} , 0.05^2 I_{2 \times 2} \right) + \frac{3}{5} \mathcal{N} \left( \begin{pmatrix} 0.75 \\
  0.75 \\ \end{pmatrix} , 0.05^2 I_{2 \times 2} \right).
\end{equation*}
We perform the Haar transform. The heatmap of the
density function is displayed in Figure~\ref{fig:heatmap1} and the plot of the Haar coefficients is
shown in Figure~\ref{fig:Haar1}. The left panel in Figure~\ref{fig:Haar1} is the plot of all the coefficients to level ten from
low resolution to high resolution. The middle one is the sorted
coefficients according to their absolute value. And the right one is
the same as the middle plot but with the abcissa in log-scale. From this we can clearly see that the
power-law decay is satisfied, and an
empirical estimation of the corresponding $\beta$ can be obtained in this case. 

\begin{figure} 
\centering
\includegraphics[scale=0.4]{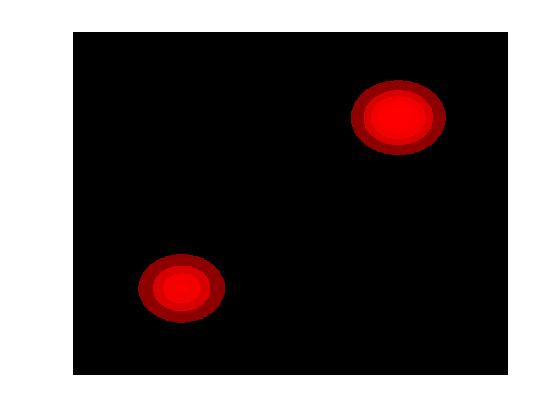}
\caption{Heatmap of the density}
\label{fig:heatmap1}
\end{figure}

\begin{figure} 
\centering
\includegraphics[scale=0.4]{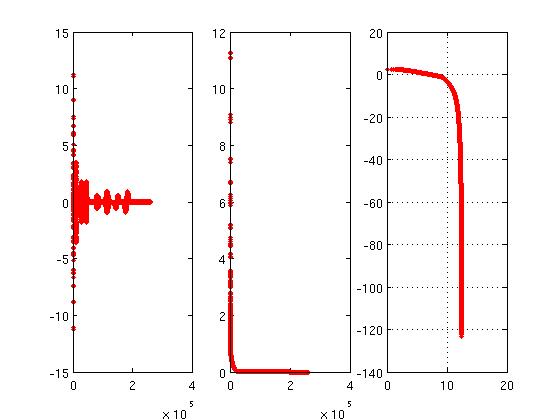}
\caption{Plots of the 2-dimensional Haar coefficients. The left panel
  is the plot of all the coefficients from low resolution to high
  resolution. The middle one is the plot of the sorted
  coefficients. And the right one is the same as the middle plot but
  with the abcissa in log scale.}
\label{fig:Haar1}
\end{figure}
\end{example}

\begin{example}
Let the three-dimensional density function be
\begin{equation*}
  \begin{pmatrix} Y_1 \\ Y_2 \\ Y_3 \\ \end{pmatrix} \sim \frac{2}{5}
  \mathcal{N} \left( \begin{pmatrix} 0.25 \\ 0.25 \\ 0.25
    \\ \end{pmatrix}, \begin{pmatrix} 0.05^2 & 0.03^2 & 0 \\ 0.03^2 &
    0.05^2 & 0 \\ 0 & 0 & 0.05^2 \\ \end{pmatrix} \right) + \frac{3}{5}
  \mathcal{N} \left( \begin{pmatrix} 0.75 \\ 0.75 \\ 0.75 \\ \end{pmatrix},
  0.05^2 I_{3 \times 3} \right).
\end{equation*}
In this example, we impose some correlation structure in one component. Haar transform is
performed and the behavior of the Haar coefficients is summarized in
Figure~\ref{fig:Haar2}. The arrangement of the plots is the same as
that in the previous example. For this three-dimensional example, the
power-law decay is still satisfied.
\begin{figure} 
\centering
\includegraphics[scale=0.4]{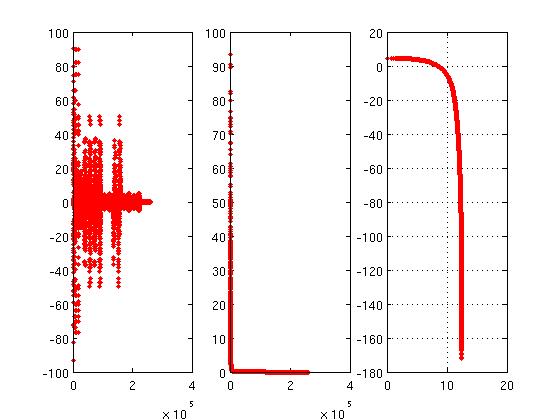}
\caption{Plots of the 3-dimensional Haar coefficients. The order of
  the plots is the same as that in Figure~\ref{fig:Haar1}.}
\label{fig:Haar2}
\end{figure}
\end{example}

\subsection{Convergence rate}
\label{subsec:spatialadaptationrate}
Assume that $f_0$ is the $p$-dimensional density function satisfy the
spatial sparsity condition (\ref{eq:powlawdecay}). Now we calculate
the convergence rate of the sieve MLE.
\begin{lemma}
\label{lemma:sparsedensity}
  Suppose $f_0$ is a $p$-dimensional density function. $g_0
  =\sqrt{f_0}$ satisfies the condition (\ref{eq:powlawdecay}). Then there
  exists a sequence of $f_I \in \Theta_I$, such that $\rho(f_0, f_I)
  \lesssim I^{-(\beta -1/2)}$, or equivalently, $\rho(f_0, f_I) \leq c
  I^{-(\beta- 1/2)}$, where $c$ may depend on $\beta$ and $p$ but not $I$. 
\end{lemma}
\begin{proof}
  Let $g_K = \sum_{k=1}^{K} <g_0,\psi_{(k)}> \psi_{(k)}$. From
  condition (\ref{eq:powlawdecay}) we have
  \begin{eqnarray}
    \label{eq:approxsparse}
   \nonumber \rho^2(f_0, g_K^2) &=& \| g_0-g_K\|_2^2 = \|
   \sum_{k=K+1}^{+ \infty} <g_0, \psi_{(k)}> \psi_{(k)} \|_2^2 \\
\nonumber &=& \sum_{k=K+1}^{+ \infty} <g_0,\psi_{(k)}>^2 \\
&\leq& C^2 \sum_{k=K+1}^{+ \infty} k^{-2\beta} \leq \frac{C^2}{2\beta -1} K^{-(2\beta-1)}.
  \end{eqnarray}
Then we normalize $g_K$ to $\tilde{g}_K$ and obtain 
\begin{eqnarray}
  \label{eq:normalizesparse}
\nonumber \rho^2(f_0, \tilde{g}_K^2) &=& \| g_0 - \tilde{g}_K
  \|_2^2 \\
\nonumber &=& \|g_0 -g_K \|_2^2  + (1-\frac{1}{\|g_K \|_2})^2
\|g_K \|_2^2 \\
\nonumber &\leq& \|g_0 -g_K \|_2^2 +1- \|g_K \|_2^2 \\
\nonumber &=&2\|g_0 -g_K \|_2^2 \\
& \leq& \frac{2C^2}{2\beta -1} K^{-(2\beta-1)}.
\end{eqnarray}
Note that given a supporting rectangle, the
positive and negative parts of the Haar basis function defined on it
can further divide the original rectangle into smaller subregions, and
the total number of such subregions is upper bounded by
$2^p$. Therefore, $2^p
K$ is the largest possible sized binary partition on which the density function $\tilde{g}_K$ is piecewise constant. Replacing $K$ in
(\ref{eq:normalizesparse}) by $I/2^p$, we get the desired result of
approximation rate.
\end{proof}
The next theorem calculates the convergence rate in this case.
\begin{theorem}
  (Application to spatial adaptation) Assume $f_0$ is the same as
  defined in Lemma (\ref{lemma:sparsedensity}). If we apply the maximum
  likelihood density estimator based on adaptive partitioning here to estimate the true density function, the convergence rate is $ n^{- \frac{\beta- 1/2}{2\beta}}
  (\log n)^{\frac{1}{2}+\frac{\beta- 1/2}{2\beta} }$.
\end{theorem}
\begin{proof}
  This follows Theorem \ref{th:convergencerate} and Lemma
  \ref{lemma:sparsedensity} directly.
\end{proof}

From the theorem we see that the convergence rate only depends on how
fast the coefficients decay as opposed to the dimension of the sample
space. Thus for large $\beta$, the density estimate is able to take
advantage of spatial sparsities to achieve fast convergence rate even
in high dimensions.
\section{Estimation of functions of bounded variation}
In image analysis, the denoised image is usually assumed to be a function of
bounded variation. Obtaining an approximation is a crucial procedure
before any downstream analysis. Currently, nonlinear approximation,
such as wavelet compression \cite{DeVore} and wavelet
shrinkage or thresholding \cite{donoho1995}, has been widely used in this field, contributing to problems including image compression and
image segmentation. In this section, we treat the denoised image as a
density function, and apply
this multivariate density estimation method to obtain the
approximation. We evaluate the performance of this method by
calculating the convergence rate when the density function is of bounded
variation (Section \ref{subsec:bdfunrate}). Actually, its performance is
comparable to that of wavelet thresholding. This point becomes clearer
in the companion papar \cite{linxiliu2015}, where under Bayesian
settings we show that
minimax convergence rate can be achieved for the space of bounded
variation (BV) up to a logarithmic term. To begin with, we briefly introduce the
space BV in Section \ref{subsec:bdfun}.

\subsection{The space BV}
\label{subsec:bdfun}
Let $\Omega= [0, 1)^2$ be a domain in
$\mathbb{R}^2$. We define the space $BV(\Omega)$ of functions of
bounded variation on $\Omega$ as follows.

For a vector $\nu \in \mathbb{R}^2$, the difference operator
$\Delta_{\nu}$ along the direction $\nu$ is defined by 
\begin{equation*}
  \Delta_{\nu}(f, y) := f( y+\nu) - f(y).
\end{equation*}
For functions $f$ defined on $\Omega$, $\Delta_{\nu}(f, y)$ is defined
whenever $y \in \Omega(\nu)$, where $\Omega(\nu) := \{ y: [y, y+\nu]
\subset \Omega \}$ and $[y, y+\nu]$ is the line
segment connecting $y$ and $y+\nu$. Denote by $e_l, l=1,2$ the two coordinate vectors in $\mathbb{R}^2$. We say that a function $f
\in L_1 (\Omega)$ is in $BV(\Omega)$ if and only if
\begin{equation*}
  V_{\Omega}(f) := \sup_{h>0} h^{-1} \sum_{l=1}^2 \| \Delta_{h e_l}
  (f, \cdot)\|_{L_1 (\Omega(h e_l) )} = \lim_{h \rightarrow 0} h^{-1} \sum_{l=1}^2 \| \Delta_{h e_l}
  (f, \cdot)\|_{L_1 (\Omega(h e_l) )}
\end{equation*}
is finite. The quantity $V_{\Omega}(f)$ is the \emph{variation} of $f$
over $\Omega$. 

$p$-dimensional bounded variation function can be defined
similarly. In \cite{donoho1993}, the author demonstrated that, in one dimension wavelet
representations of bounded variation balls are optimal. The result
of optimality can be extended to two-dimensional cases as well. Given
the fact that the density functions defined on binary
partitions are essentially equivalent to the Haar bases, it
is natural to ask what the convergence rate will be if we apply this
multivariate density estimator to $BV(\Omega)$. This result will be
presented in the section below. 
\subsection{Convergence rate} 
\label{subsec:bdfunrate}
In this section, we still use the Haar basis defined in Section
\ref{subsec:Haarbasis1}. Let $\Lambda$ be the set of indices for
the wavelet basis. Each element in $\Lambda$ is a pair of scale
and location parameters. We will denote by $\Sigma_N$ the spaces
consisting of $N$-term approximation in the Haar system, in other words,
\begin{equation*}
  \Sigma_N := \{ \sum_{\lambda \in E} c_{\lambda} \psi_{\lambda}: E
  \subset \Lambda, |E| \leq N \},
\end{equation*}
where $|E|$ denotes the cardinality of the discrete set $E$.

First, we cite a theorem from \cite{cohen1999}. It provides a result
on the approximation rate to a function of bounded variation by $\Sigma_N$.
\begin{lemma}
  If $f \in BV(\Omega)$ has mean value zero on $\Omega$, we have
  \begin{equation}
    \label{eq:bdapproximation}
    \inf_{g \in \Sigma_N} \| f - g \|_{L_2(\Omega)} \leq C N^{-1/2}
    V_{\Omega} (f),
  \end{equation}
with $C=2592(3\sqrt{5} +\sqrt{3})$.
\end{lemma}

Assume $f_0$ is a density function on $\Omega$ of bounded
variation. By subtracting the mean, we can always assume that $\sqrt{f_0}$ has
mean value zero over $\Omega$. For the square root of $f_0$, applying the
lemma above, we can find an $N$-term approximation $g$ in the Haar system,
such that $\|\sqrt{f_0} -g \|_{L_2(\Omega)} \lesssim
N^{-1/2}$. Translating this inequality into the size of partition, we
reach the conclusion that for a density function in $BV(\Omega)$, we
can find an approximation in $\Theta_I$, such that $\rho(f_0, f_I) \lesssim
I^{-1/2}$. Now we are ready to state the result of the convergence
rate.

\begin{theorem}
  Assume that $f_0 \in BV(\Omega)$. If we apply the multivariate
  density estimator based on adaptive partitioning here to estimate $f_0$, the convergence rate is
  $n^{-1/4} (\log n)^{3/4}$.
\end{theorem}
\begin{proof}
  The proof follows Theorem \ref{th:convergencerate} and the previous
  approximation result directly.
\end{proof}

\section{Application to variable selection}
\label{sec:variableselection}
For high dimensional data analysis, selecting significant variables
greatly contributes to simplifying the model, improving model
interpretability, and reducing overfitting. In the context of density estimation, the variable
selection problem is
formulated as follows. Assume $f_0$ is a $p$-dimensional density
function and it only depends on $\tilde{p}$
variables, but we do not know  in
advance which $\tilde{p}$ variables. We apply the multivariate density
estimation method here to the $p$-dimensional density. The essential part of our
method is to learn a partition of the support of the true
density function, and
then to estimate the density on each subregion separately. Because the
true density lies in a $\tilde{p}$-dimensional subspace, we may surmise
that the corresponding convergence rate only depends on the effective
dimension. The goal of this section is to carry out exact
calculations to reveal that this is indeed the case. Here,
we consider
a class of density functions satisfying certain type of continuity. We
first provide a mathematical description of this density class in
Section \ref{subsec:functionclass}. This description still depends on the Haar transform of the true density function. However, an alternative
construction of the high-dimensional Haar basis via tensor product
is introduced in Section \ref{subsec:Haarbasis2} because of some
technical issue. We provide the result
on the explicite convergence rate in Section \ref{subsec:variableselectionrate}.

\subsection{Tensor Haar basis}
\label{subsec:Haarbasis2}
In one dimension, the Haar wavelet's mother wavelet function $\psi$
and its scaling function $\phi$ are the same as those defined in
Section \ref{subsec:Haarbasis1}. For any $l \in \mathbb{N}$ and $0\leq k < 2^l$, the Haar function is
\begin{equation*}
  \psi^l_k (y) = 2^{l/2} \psi(2^l y -k).
\end{equation*}
Then Haar basis \textbf{$\Psi$} is the collection of all Haar
functions together with the scale function. Namely,
\begin{equation*}
  \textbf{$\Psi$} = \{ \phi \} \cup \{ \psi^l_k, l \in \mathbb{N},
  0\leq k < 2^l \}.
\end{equation*}
It is an orthonormal basis for Hilbert space $L^2([0,1])$. 

Turning to high-dimensional settings, we can obtain an orthonormal
basis for $L^2 ([0,1]^p)$ by using the fact that the Hilbert space $L^2 ([0,1]^p)$ is
isomorphic to the tensor product of $p$ one-dimensional spaces. In
detail, if $\mathcal{X}_1, \cdots, \mathcal{X}_p$ are $p$ copies of
$L^2 ([0,1])$ and \textbf{$\Psi_1$}, $\cdots$, \textbf{$\Psi_p$} are
Haar bases of these spaces respectively, then $L^2 ([0,1]^p)$ is isomorphic to $\bigotimes_{i=1}^p
\mathcal{X}_i$. Define tensor Haar basis \textbf{$\Psi$} by
\begin{equation*}
  \textbf{$\Psi$} = \{\psi: \psi= \prod_{i=1}^p \psi_i, \psi_i \in \textbf{$\Psi_i$} \}.
\end{equation*}
From the property of tensor product of Hilbert spaces, we know that
\textbf{$\Psi$} is an orthonormal basis for $L^2 ([0,1]^p)$.

\subsection{Mixed-H\"older continuity}
\label{subsec:functionclass}
In this section, we assume the density function satisfies the \emph{Mixed-H\"older} condition.
A similar condition first appeared in \cite{stromberg1998}. The author showed that tensor Haar basis with large
support is efficient in representing certain type of functions on
$[0,1]^p$. More precisely, if the function
satisfies the \emph{bounded mixed variation} condition, then it can be
approximated to error $\epsilon>0$ using no more than $
\frac{1}{\epsilon} (\log (\frac{1}{\epsilon}) )^{p-1}$ terms, and the
volume of the supporting rectangle of
the each wavelet basis function involved in the approximation is
greater than $\epsilon$. However, the result is restrictive in
application since the
mixed derivative is not rotationally invariant. In \cite{gavish2012},
the authors extended the previous approximation scheme to matrix. A significant improvement is that, in their paper, the bounded mixed variation
condition is replaced by the Mixed-H\"older condition, which is more
natural and accessible. For two-dimensional discrete analyses of matrices,
they show that the approximation is still efficient for this class of matrices, and the Mixed-H\"older
condition is connected to the decay rate of wavelet coefficients. The
idea of controlling the decay rate of wavelet coefficients will be further
developed here. It leads to the characterization of
the density class under consideration.

For any $f \in \mathcal{F}$, $\sqrt{f} \in L^2 ([0,1]^p)$. Therefore, we can expand $\sqrt{f}$ with respect to
the tensor Haar basis. Let $g= \sqrt{f}$. Then 
\begin{equation*}
g = \sum_{\psi \in \textbf{$\Psi$}} <g,
\psi> \psi,\ \mbox{where} <g,\psi>=\int_{\Omega} g(y) \psi(y) dy.
\end{equation*} 
For each tensor Haar function $\psi$, let $R(\psi)$ denote its supporting
rectangle. The density class under consideration is defined to be
\begin{equation*}
  \mathcal{F}_H = \{ f \in \mathcal{F}: |<\sqrt{f}, \psi>| \leq C |R(\psi)|^{\alpha+1/2}
  \mbox{ for all } \psi \in \textbf{$\Psi$} \},
\end{equation*}
where $C$ is a constant which may depend on $f$, $|R(\psi)|$ denotes the volume of the
rectangle and $\alpha$ is a positive constant.

Next, we provide several examples to illustrate how this
condition relates to the Mixed-H\"older continuity.
\begin{example}
A real-valued function $f$ on $\mathbb{R}$ is \emph{H\"older continuous}, if
there exist nonnegative constants $C$ and $\alpha \in (0, 1]$, such that $|f(x) - f(y)|
\leq C |x-y|^{\alpha}$, for all $x, y \in \mathbb{R}$.
If square root of the true density function $f_0$ is H\"older continuous for some
constants $C, \alpha$, then for any Haar basis function $\psi$,
$|<\sqrt{f_0}, \psi>| \leq C |R(\psi)|^{\alpha+1/2}$. 
\begin{proof}
  From H\"older continuity, we know that for any $x, y \in [0,1]$,
  $|\sqrt{f_0(x)}- \sqrt{f_0(y)}| \leq C |x-y|^{\alpha}$. For any
  point $x_0 \in R$, we have
  \begin{eqnarray*}
    |<\sqrt{f_0}, \psi>| ^2 &=& \left( \int_{R} \sqrt{f_0(x)} \psi(x) dx
    \right)^2 \\
    &=& \left( \int_{R} \left(\sqrt{f_0(x)} -\sqrt{f_0(x')} \right) \psi(x) dx \right)^2 \\
&\leq& \int_{R} \left(\sqrt{f_0(x)} -\sqrt{f_0(x')} \right)^2 dx \cdot \int_{R} \psi(x)^2 dx \\
&\leq& C^2 \int_{R} |x-x_0|^{2\alpha} dx \\
&=& C^2 |R|^{2\alpha+1}.
  \end{eqnarray*}
The desired results follows.
\end{proof}
\end{example}

\begin{example}
A real-valued function $f$ on $\mathbb{R}^2$ is called
Mixed-H\"older continuous for some nonnegative constant $C$ and
$\alpha \in (0, 1]$, if for any $(x_1, y_1), (x_1, y_2) \in
\mathbb{R}^2$,
\begin{equation*}
  |f(x_2, y_2) - f(x_2, y_1) - f(x_1, y_2) + f(x_1, y_1) | \leq C |x_1
  -x_2|^{\alpha} |y_1 -y_2 |^{\alpha}.
\end{equation*}
If square root of the two-dimensional true density function $f_0$ is
Mixed-H\"older continuous for some constants $C, \alpha$, then $f_0
\in \mathcal{F}_H$. The proof is similar to the one-dimensional case.
\end{example}

More generally, this type of continuity condition can be extended to
high-dimensional cases \cite{stromberg1998}. The corresponding density functions also belong to
the space $\mathcal{F}_H$. Therefore, in this section, we use the more
general condition
$|<\sqrt{f}, \psi>| \leq C |R(\psi)|^{\alpha+1/2}$ for all $\psi$ to characterize the density class.

\subsection{Convergence rate} 
\label{subsec:variableselectionrate}
First, we provide a result on the rate at which the approximation
error decreases to zero as the complexity of the approximating spaces increases.
\begin{lemma}
  \label{th:approximationrate}
$\mathcal{F}_H$ and $\Theta_I$ are defined as above. For any $f_0
\in \mathcal{F}_H$, there exists a sequence of $f_I \in \Theta_I$, such
that $\rho(f_0, f_I) \lesssim I^{-\alpha/p} (\log I)^{p/2}$, where $\alpha$ is as defined
in the Section \ref{subsec:functionclass}, and $p$ is the dimension of
the Euclidean space. 
\end{lemma}
\begin{proof}
\label{th:approxrate}
  Let $g_0 = \sqrt{f_0}$. We can expand $g_0$ with respect to the tensor
  Haar basis. The expansion can be written as $g_0 = \sum_{\psi} <g_0,
  \psi> \psi$.

Let $g_{\epsilon} = \sum_{\psi: |R(\psi)| > \epsilon} <g_0, \psi>
\psi$. Then $g_{\epsilon}$ is an approximation to $g_0$ obtained by requiring that the volumes of the supporting rectangles of
the involved wavelet basis functions are greater
than $\epsilon$. We will derive an approximation rate as a function of
$\epsilon$ first, and then convert the lower bound on the volume to an upper bound
on the size of the partition. This yields an approximation rate as a function of the size of the partition.
Note that $g_{\epsilon}$ is not a density function, but it is easier
to work with. Let $\tilde{g}_{\epsilon} = g_{\epsilon}/ \|g_{\epsilon}
\|_2$ be the normalization of $g_{\epsilon}$. The upper bounds for the
approximation errors $\rho(f_0, g_{\epsilon} ^2)$ and $\rho(f_0,
\tilde{g}_{\epsilon}^2)$ will be derived successively.

Before delving into the proof, we introduce some notations first.
For each  supporting rectangle $|R(\psi)|$, the lengths of its edges
should be powers of $1/2$. We may assume that $\psi=\prod_{i=1}^p
\psi_i$, and for each $\psi_i$ the length of its supporting interval
is $(1/2)^{l_i}$. Let $\mathcal{R}^{l_1, \cdots, l_p}$ denote the
collection of the rectangles for which the lengths of the edges
are $(1/2)^{l_1}, \cdots, (1/2)^{l_p}$.

Recall that $f_0$ satisfies the condition
\begin{equation}
\label{eq:densityfuncond}
  | <\sqrt{f_0}, \psi>| \leq C |R(\psi)|^{\alpha + 1/2} \mbox{ for
    all } \psi.
\end{equation}
Then,
\begin{eqnarray}
\label{eq:approxI}
\nonumber \rho^2(f_0, g_{\epsilon}^2 ) &=& \| g_0 - g_{\epsilon} \| ^2 = \|
 \sum_{\psi: |R(\psi)| < \epsilon } <g_0, \psi> \psi \|_2^2\\ 
\nonumber &=& \sum_{\psi: |R(\psi)| <\epsilon} <g_0, \psi>^2 \\ 
\nonumber &\leq& C^2 \sum_{\psi: |R(\psi)|< \epsilon } |R(\psi)|^{2\alpha+1} \\
&\leq& 2^p C^2 \sum_{l_1, \cdots, l_p} \sum_{R \in \mathcal{R}^{l_1,
    \cdots, l_p}, |R| <\epsilon } |R|^{2\alpha+1}.
\end{eqnarray}
The last inequality follows from the fact that, given a supporting
rectangle, there are at most $2^p$
basis functions defined on it. Let $N= \lceil
\log_{\frac{1}{2}} \epsilon \rceil$,
\begin{eqnarray}
  \label{eq:approxII}
 \nonumber  (\ref{eq:approxI}) &=& 2^p c^2 \sum_{l_1 + \cdots + l_p \geq N }
  \sum_{R \in \mathcal{R}^{l_1, \cdots, l_p}} |R|^ {2\alpha+1} \\
\nonumber &=& 2^p C^2 \sum_{l_1+ \cdots+ l_p \geq N} (\frac{1}{2})^{2\alpha(l_1
  +\cdots +l_p)} \sum_{R \in \mathcal{R}^{l_1, \cdots, l_p}} |R| \\
&=& 2^p C^2 \sum_{l_1+ \cdots +l_p \geq N} (\frac{1}{2})^{2\alpha(l_1
  +\cdots +l_p)}. 
\end{eqnarray}
The last equality is obtained by plugging in $\sum_{R \in \mathcal{R}^{l_1,
    \cdots, l_p}} |R| =1$.
Note that
\begin{eqnarray}
  \label{eq:approxIII}
 & & \sum_{l_1+\cdots +l_p \geq N} ( \frac{1}{2} ) ^{2\alpha(l_1+\cdots
    +l_p)} \\
\nonumber &\leq& \sum_{l_1=0}^{N} \sum_{l_2=0}^{N-l_1} \cdots
\sum_{l_p=N-(l_1+\cdots +l_{p-1})}^{+ \infty}
(\frac{1}{2})^{2\alpha(l_1 +\cdots +l_p)} \\
\nonumber & &+ \sum_{l_1=0}^N \sum_{l_2=0}^{N-l_1} \cdots \sum_{l_{p-1}=N-(l_1
  +\cdots +l_{p-2})} ^{+\infty} \sum_{l_p=0}^{+\infty}
(\frac{1}{2})^{2\alpha(l_1 +\cdots +l_p)}  \\
\nonumber & &+ \cdots \\
\nonumber & &+ \sum_{l_1=N}^{+\infty} \sum_{l_2=0}^{+\infty} \cdots
\sum_{l_p=0}^{+\infty} (\frac{1}{2})^{2\alpha(l_1 +\cdots +l_p)} \\
\nonumber &\leq& (N+1)^{p-1} \frac{ (\frac{1}{2})^{2\alpha
    N}}{1-2^{-2\alpha}} + (N+1)^{p-2} \frac{ (\frac{1}{2})^{2\alpha
    N}}{(1-2^{-2\alpha})^2} + \cdots + \frac{ (\frac{1}{2})^{2\alpha
    N} } {(1-2^{-2\alpha})^p} \\
\nonumber &=& (\frac{1}{2})^{2\alpha N} \frac{(N+1)^p -
  (1-2^{-2\alpha})^{-p}} {(N+1)(1-2^{-2\alpha}) -1} \\
\nonumber &\leq& C' \epsilon^{2\alpha} (\log_{\frac{1}{2}} \epsilon)^p.
\end{eqnarray}
From this, we know that
\begin{equation}
  \label{eq:approxIV}
  \rho^2(f_0, g_{\epsilon}^2 ) = \|g_0 - g_{\epsilon} \|_2^2 \leq 2^p
  C' C^2 \epsilon^{2\alpha} (\log_{\frac{1}{2}} \epsilon)^p.
\end{equation}

We normalize $g_{\epsilon}$ to $\tilde{g}_{\epsilon}$, then
\begin{eqnarray*}
  \rho^2(f_0, \tilde{g}_{\epsilon}^2) &=& \| g_0 - \tilde{g}_{\epsilon}
  \|_2^2 \\
&=& \|g_0 -g_{\epsilon} \|_2^2  + (1-\frac{1}{\|g_{\epsilon} \|_2})^2
\|g_{\epsilon} \|_2^2 \\
&\leq& \|g_0 -g_{\epsilon} \|_2^2 +1- \|g_{\epsilon} \|_2^2 \\
&=&2\|g_0 -g_{\epsilon} \|_2^2 .
\end{eqnarray*}
The last equality is obtained by using $\|g_0 - g_{\epsilon} \|_2^2 +
\|g_{\epsilon} \|_2^2 =\|g_0 \|_2^2 =1$.Therefore, 
\begin{eqnarray}
  \label{eq:approxrate}
  \rho^2(f_0, \tilde{g}_{\epsilon}^2 ) = \|g_0 - \tilde{g}_{\epsilon}
  \|_2^2 \leq 2^p C'' C^2 \epsilon^{2\alpha} (\log_{\frac{1}{2}} \epsilon)^p,
\end{eqnarray}
where $C''$ is a constant.

Next, we will convert the lower bound on the volume of the supporting
rectangles to an upper bound on the size of the partition, and derive
the approximation rate in terms of the latter one.

If we require the volumes of the supporting rectangles be greater
than $\epsilon$, then the lengths of the edges can not be smaller than
$2^{-\lfloor \log_{\frac{1}{2}} \epsilon \rfloor}$. The size of the
partition supporting $\tilde{g}_{\epsilon}$ can be bounded by $2^p 2^{p
  \log_{\frac{1}{2}} \epsilon} = 2^p \epsilon^{-p}$. There is a
coefficient $2^p$ in front. This is the case because given a supporting rectangle, the
positive and negative parts of the tensor Haar basis defined on it
will further divide the original rectangle into smaller subregions and the
number of such subregions is at most $2^p$.

Given the size of the partition $I$, we can determine $\epsilon$ by
solving $2^p \epsilon^{-p} =I$ and define $\tilde{g}_{\epsilon} \in
\Theta_I$ as above. Then from (\ref{eq:approxrate}) we reach a conclusion that $\tilde{g}_{\epsilon}$ is an approximation satisfying
$\rho(f_0, \tilde{g}_{\epsilon}^2) \lesssim I^{-\alpha/p} (\log I)^{p/2}$. This finishes
the proof.
\end{proof}

\begin{theorem}
\label{th:variableselectionrate}
Assume that $f_0 \in \mathcal{F}_H$ is a $p$-dimensional density function. It only
depends on $\tilde{p}$ arguments which are not specified in
advance. If we apply
the multivariate density estimation method to this problem, the
convergence rate is $n^{-\frac{\alpha}{2\alpha +\tilde{p}} } (\log
n)^{1+\frac{\tilde{p} (\tilde{p}-1)/2}{2\alpha+\tilde{p}}}$.
\end{theorem}
\begin{proof}
  Because $f_0$ lies in a $\tilde{p}$-dimensional subspace, the decay
  rate of the approximation error only depends on $\tilde{p}$ instead of $p$. Then the
  result follows Theorem \ref{th:convergencerate} and Theorem
  \ref{th:approximationrate}.
\end{proof}

From this theorem, we learn that only the effective dimension affects the rate of our method.
This implies that in extreme cases of $\tilde{p} \ll p$, our method can still achieve stable
performances. The advantage of our method is demonstrated by the following facts: in
the partition learning stage, it can quickly
restrict our attention to those relevant variables. Ideally, it can estimate the
density as a function of the effective variables alone, although they are not specified in advance.

\appendix

\section*{Acknowledgements}
The authors would like to thank Emmanuel Cand\`es, Matan Gavish and Xiaotong Shen for helpful discussions.

\bibliographystyle{imsart-nameyear}
\bibliography{ref}

\end{document}